\documentclass[10pt,draft]{amsart}
\usepackage{amsmath}
\usepackage{amsthm}
\usepackage{amsfonts}  
\usepackage{latexsym}
\usepackage[mathscr]{eucal}

\DeclareMathAlphabet\mathoo{U}{eur}{b}{n}
 \DeclareMathOperator*{\esssup}{ess\,sup}

\theoremstyle{plain}
\newtheorem{theorem}{Theorem}[section]
\newtheorem{proposition}[theorem]{Proposition}
\newtheorem{corollary}[theorem]{Corollary}

\theoremstyle{remark}
\newtheorem{remark}{Remark}
 \numberwithin{equation}{section}

\begin{document}

\title[Inverse-closedness of the set of integral operators]
 {Inverse-closedness of the set\\ of integral operators\\
with $L_1$-continuously varying kernels}

\author[V.G. Kurbatov]{V.G. Kurbatov}

\address{Department of Mathematical Physics\\
Voronezh State University\\
1, Universitetskaya Square\\
Voronezh 394036\\
Russia} \email{kv51@inbox.ru}

\thanks{This work was supported by the Ministry of Education and Science of the Russian
Federation under state order No. 1306.}

\author{V.I. Kuznetsova}
\address{Department of Applied Mathematics and Mechanics\\
Voronezh State Technical University\\
14, Mos\-cow Avenue\\
Voronezh 394026\\
Russia} \email{kv57@bk.ru}


\date{\today}

\begin{abstract}
Let $N$ be an integral operator of the form
\begin{equation*}
\bigl(Nu\bigr)(x)=\int_{\mathbb R^c}n(x,x-y)\,u(y)\,dy
\end{equation*}
acting in $L_p(\mathbb R^c)$ with a measurable kernel $n$ satisfying the estimate
\begin{equation*}
|n(x,y)|\le\beta(y),
\end{equation*}
where $\beta\in L_1$. It is proved that if the function $t\mapsto n(t,\cdot)$ is
continuous in the norm of $L_1$ and the operator $\mathbf1+N$ has an inverse, then
$(\mathbf1+N)^{-1}=\mathbf1+M$, where $M$ is an integral operator possessing the same
properties.
\end{abstract}

\maketitle

\section{Introduction}
A class $\mathoo A$ of linear operators is called inverse-closed if the inverse to any
operator from $\mathoo A$ also belongs to $\mathoo A$. Usually, an inverse-closed class
forms a subalgebra of the algebra of all bounded operators. The investigation of
inverse-closed subalgebras (full subalgebras) had its origin in Wiener's theorem on
absolutely convergent Fourier
series~\cite{Wiener,Bochner-Phillips,Gel'fand-Raikov-Shilov}. Wiener's theorem implies
that if the operator $\mathbf1+N$, where $N$ is an operator of convolution with a
summable function, is invertible, then $(\mathbf1+N)^{-1}=\mathbf1+M$, where $M$ is also
an operator of convolution with a summable function. For more recent results on
inverse-closed classes, see~\cite{Balan-Krishtal}--\cite{Sun11} and references therein.

This paper deals with the integral operator in $L_p(\mathbb R^c,\mathbb E)$, $1\le
p\le\infty$, of the form
\begin{equation*}
\bigl(Nu\bigr)(x)=\int_{\mathbb R^c}n(x,x-y)\,u(y)\,dy.
\end{equation*}
It is assumed that $\mathbb E$ is a finite-dimensional Banach space, the values $n(x,y)$
of the kernel $n$ are bounded linear operators acting in $\mathbb E$, and the kernel $n$
is measurable and satisfies the estimate
\begin{equation}\label{e:beta}
\Vert n(x,y)\Vert\le\beta(y),\tag{$*$}
\end{equation}
where $\beta\in L_1$. The main result of the paper (Theorem~\ref{t:fin}) states that if
the function $t\mapsto n(t,\cdot)$ is continuous in the $L_1$-norm and the operator
$\mathbf1+N$ has an inverse, then $(\mathbf1+N)^{-1}=\mathbf1+M$, where $M$ is an
integral operator possessing the same properties.

The idea of the proof consists of a combination of two results. The fist
result~\cite{Beltita,Farrell-Strohmer,Kurbatov99,Kurbatov01} states that if
estimate~\eqref{e:beta} holds for the kernel $n$ of the operator $N$ and $\mathbf1+N$ is
invertible in $L_p(\mathbb R^c,\mathbb E)$, then $(\mathbf1+N)^{-1}=\mathbf1+M$, where
$M$ is an integral operator with a kernel $m$ satisfying estimate~\eqref{e:beta} as well.
This easily implies that the fact of the invertibility of $\mathbf1+N$ and the kernel $m$
do not depend on $p$. On the other hand, it was known~\cite{Kurbatov99} that (for a wide
class of operators $T$) if `coefficients' of $T:\,L_p\to L_p$ vary continuously in an
abstract sense, then $T^{-1}$ has the same property (provided $T^{-1}$ exists). Here the
abstract continuity means that $S_hTS_{-h}$, where $\bigl(S_hu\bigr)(x)=u(x-h)$,
continuously depends upon $h$ on subspaces of compactly supported (in a uniform sense)
functions. If $p=\infty$ such abstract continuity of the integral operator $N$ exactly
means that the function $t\mapsto n(t,\cdot)$ is continuous in the $L_1$-norm. Technical
difficulties of the proof mostly consist of the correct usage of the Lebesgue integral.
As a generalization of the main result, me show (Theorem~\ref{t:fin2}) that an inverse to
a difference-integral operator with `continuous' coefficients also has `continuous'
coefficients.

The paper is organized as follows. General facts concerning the Lebesgue integral are
recalled in Section~\ref{s:Lebesgue integral}. In Section~\ref{s:N_1}, we describe some
properties of the class of integral operators majorized by a convolution with a function
$\beta\in L_1$. In Section~\ref{e:Cu and C}, we discuss operators whose coefficients
(kernels) vary continuously in an abstract sense. In Section~\ref{s:CN_1}, we prove
Theorem~\ref{t:fin} which is the main result of this paper. In Section~\ref{e:h}, we show
that the integral operator $N$ considered possesses the property of local compactness.
This fact allows us to generalize Theorem~\ref{t:fin} to a class of difference-integral
operators (Section~\ref{s:D+N}).

\section{General notation and the Lebesgue integral}\label{s:Lebesgue integral}
Let $X$ and $Y$ be Banach spaces. We denote by $\mathoo B(X,Y)$ the space of all bounded
linear operators acting on $X$ to $Y$. If $X=Y$ we use the brief notation $\mathoo B(X)$.
We denote by $\mathbf1\in\mathoo B(X)$ the identity operator.

As usual, $\mathbb Z$ is the set of all integers and $\mathbb N$ is the set of all
positive integers.

Let $c\in\mathbb N$. Unless otherwise is explicitly stated, the linear space $\mathbb
R^c$ is considered with the Euclidian norm $|\cdot|$. For $x\in\mathbb R^c$ and $r>0$, we
denote by $B(x,r)$ the open ball $\{\,y\in\mathbb R^c:\,|y-x|<r\,\}$ with centre at $x$
and radius $r$.

We denote by $\mu$ the Lebesgue measure on $\mathbb R^c$. We accept that a measurable
function~\cite{Bourbaki Int'egration,Edwards,Hewitt-Ross} may be undefined on a set of
measure zero. We say that $E\subseteq\mathbb R^c$ is a set of \emph{full measure} if its
complement has measure zero.

Let $E$ be a measurable subset of $\mathbb R^c$. The point $x\in E$ is
called~\cite[ch.~1, \S~2]{Stein} a \emph{point of density} of $E$ if
\begin{equation*}
\lim_{r\to+0}\frac{\mu\bigl(E\cap B(x,r)\bigr)}{\mu\bigl(B(x,r)\bigr)}=1.
\end{equation*}
 \begin{proposition}[{\rm Lebesgue's density theorem \cite[ch. 1, \S~2, Proposition 1]{Stein}}\bf]\label{p:Stein:1:1}
Almost every point of a measurable set $E\subseteq\mathbb R^c$ is a point of density of
$E$.
 \end{proposition}
 \begin{proposition}[{\rm Lusin's theorem~\cite[ch.~4, \S~5 Proposition 1]{Bourbaki Int'egration}}\bf]\label{p:mes func}
Let $X$ be a Banach space. A function $f:\,\mathbb R^c\to X$ is measurable if and only if
for any compact set $K\subset\mathbb R^c$ and any $\varepsilon>0$ there exists a compact
set $K_1\subseteq K$ such that $\mu(K\setminus K_1)<\varepsilon$ and the restriction of
$f$ to $K_1$ is continuous.
 \end{proposition}

Let $\mathbb E$ be a fixed finite-dimensional Banach space with the norm $|\cdot|$. We
denote by $\mathscr L_p=\mathscr L_p(\mathbb R^c,\mathbb E)$, $1\le p<\infty$, the space
of all measurable functions $u:\,\mathbb R^c\to\mathbb E$ bounded by the semi-norm
\begin{equation*}
\Vert u\Vert=\Vert u\Vert_{L_p}=\Bigl(\int_{\mathbb R^c}|u(x)|^p\,dx\Bigr)^{1/p},
\end{equation*}
and we denote by $\mathscr L_\infty=\mathscr L_\infty(\mathbb R^c,\mathbb E)$ the space
of all measurable essentially bounded functions $u:\,\mathbb R^c\to \mathbb E$ with the
semi-norm
\begin{equation*}
\Vert u\Vert=\Vert u\Vert_{L_\infty}=\esssup|u(x)|.
\end{equation*}
Finally, we denote by $L_p=L_p(\mathbb R^c,\mathbb E)$, $1\le p\le\infty$, the Banach
space of all classes of functions $u\in\mathscr L_p$ with the identification almost
everywhere. For more details, see~\cite{Bourbaki Int'egration,Edwards,Hewitt-Ross}.
Usually they do not distinguish the spaces $\mathscr L_p$ and $L_p$. For our purposes it
is not always convenient. Both the semi-norm on $\mathscr L_p$ and the induced norm
on~$L_p$ are called $L_p$-\emph{norms}.

 \begin{proposition}[{\rm Lebesgue's theorem~\cite[ch.~4, \S~3, 7, Theorem 6]{Bourbaki Int'egration}}\bf]\label{p:Lebesgue}
Let $p<\infty$, $u_i\in\mathscr L_p(\mathbb R^c,\mathbb E)$ converges almost everywhere
to a function $u$, and there exists a nonnegative function $g\in\mathscr L_p(\mathbb
R^c,\mathbb R)$ such that $|u_i(x)|\le g(x)$ for almost all $x$ and all $i$. Then
$u\in\mathscr L_p$ and $u_i$ converges to $u$ in $L_p$-norm.
 \end{proposition}

 \begin{proposition}\label{p:abs ser}
Let $1\le p\le\infty$, $u_i\in\mathscr L_p(\mathbb R^c,\mathbb E)$, and the series
$\sum_{i=1}^\infty u_i$ converges absolutely, i.e. $\sum_{i=1}^\infty\Vert
u_i\Vert_{L_p}<\infty$. Then the series $\sum_{i=1}^\infty u_i(x)$ converges absolutely
at almost all $x$, the function $s(x)=\sum_{i=1}^\infty u_i(x)$ belongs to $\mathscr
L_p$, and the series $\sum_{i=1}^\infty u_i$ converges to $s$ in $L_p$-norm.
 \end{proposition}
 \begin{proof}
For the case $p<\infty$, e.g. see~\cite[ch.~4, \S~3, 3, Proposition 6]{Bourbaki
Int'egration}. The case $p=\infty$ is evident.
 \end{proof}

 \begin{proposition}[{\rm Fubini's theorem,~\cite[ch.~5, \S~8, 4]{Bourbaki Int'egration}}\bf]\label{p:Fubini}
Let $X$ be an arbitrary Banach space.

If $n\in\mathscr L_1(\mathbb R^c\times\mathbb R^c,X)$, then for almost all $x\in\mathbb
R^c$ the function
\begin{equation*}
y\mapsto n(x,y)
\end{equation*}
is defined for almost all $y\in\mathbb R^c$ and belongs to $\mathscr L_1(\mathbb
R^c,X)${\rm;} the function
\begin{equation*}
x\mapsto\int_{\mathbb R^c}n(x,y)\,dy
\end{equation*}
is defined for almost all $x\in\mathbb R^c$ and belongs to $\mathscr L_1(\mathbb
R^c,X)${\rm;} and
\begin{equation*}
\iint_{\mathbb R^c\times\mathbb R^c}n(x,y)\,dx\,dy=\int_{\mathbb R^c}\Bigl(\int_{\mathbb
R^c}n(x,y)\,dy\Bigr)\,dx.
\end{equation*}

If $n:\,\mathbb R^c\times\mathbb R^c\to X$ is measurable and
\begin{equation*}
\int_{\mathbb R^c}\Bigl(\int_{\mathbb R^c}\Vert n(x,y)\Vert\,dy\Bigr)\,dx<\infty,
\end{equation*}
then $n\in\mathscr L_1(\mathbb R^c\times\mathbb R^c,X)$.
 \end{proposition}

 \begin{corollary}\label{c:Fubini}
A measurable subset $E\subset\mathbb R^c\times\mathbb R^c$ has measure zero if and only
if for almost all $x\in\mathbb R^c$ the set $E_x=\{\,y\in\mathbb R^c:\,(x,y)\in E\,\}$
has measure zero.
 \end{corollary}

 \begin{proposition}\label{p:sequence in L_1} Let $X$ be a Banach space
and a sequence $u_k\in\mathscr L_1(\mathbb R^c,X)$ converge to $u_0\in\mathscr
L_1(\mathbb R^c,X)$ in norm. Then there exists a subsequence $u_{k_i}$ that converges to
$u_0$ almost everywhere.
 \end{proposition}
 \begin{proof}
This is a consequence of~\cite[ch.~4, \S~3, 4, Theorem 3]{Bourbaki Int'egration} and
Proposition~\ref{p:abs ser}.
 \end{proof}

\section{The class $\mathoo N_1$}\label{s:N_1}
We denote by $\mathoo N_1=\mathoo N_1(\mathbb R^c,\mathbb E)$ the set of all measurable
functions $n:\,\mathbb R^c\times\mathbb R^c\to\mathoo B(\mathbb E)$ satisfying the
property: there exists a function $\beta\in\mathscr L_1(\mathbb R^c,\mathbb R)$ such that
for almost all $(x,y)\in\mathbb R^c\times\mathbb R^c$
\begin{equation}\label{e:est via beta}
\Vert n(x,y)\Vert\le\beta(y).
\end{equation}
For convenience (without loss of generality), we assume that $\beta$ is defined
everywhere. Kernels of the class $\mathoo N_1$ and the operators induced by them were
considered in~\cite{Beltita,Farrell-Strohmer}, \cite[\S~5.4]{Kurbatov99},
and~\cite{Kurbatov01}. In order to show that the notation used
in~\cite[\S~5.4]{Kurbatov99} and~\cite{Kurbatov01} is equivalent to the notation used in
the present paper, we note the following Proposition.

 \begin{proposition}\label{p:N acts in L_p}
The function $n:\,\mathbb R^c\times\mathbb R^c\to\mathoo B(\mathbb E)$ is measurable if
and only if the functions $n_1(x,y)=n(x,x-y)$ is measurable.
 \end{proposition}
 \begin{proof}
The proof can be obtained as a word to word repetition of the proof
of~\cite[Lemma~4.1.5]{Kurbatov99}.
 \end{proof}


 \begin{proposition}[{\rm\cite[Proposition~5.4.3]{Kurbatov99}}\bf]\label{p:5.4.3}
For any $n\in\mathoo N_1(\mathbb R^c,\mathbb E)$, the operator
 \begin{equation}\label{e:operator N}
\bigl(Nu\bigr)(x)=\int_{\mathbb R^c}n(x,x-y)\,u(y)\,dy
 \end{equation}
acts in $L_p(\mathbb R^c,\mathbb E)$ for all $1\le p\le\infty$. More precisely, for any
$u\in\mathscr L_p(\mathbb R^c,\mathbb E)$ the function $y\mapsto n(x,x-y)\,u(y)$ is
integrable for almost all $x$, and the function $Nu$ belongs to $\mathscr L_p(\mathbb
R^c,\mathbb E)${\rm;} if $u_1$ and $u_2$ coincide almost everywhere, then $Nu_1$ and
$Nu_2$ also coincide almost everywhere. Besides,
 \begin{equation}\label{e:norm of N}
\Vert N:\,L_p\to L_p\Vert\le\Vert\beta\Vert_{L_1}.
 \end{equation}
 \end{proposition}

We denote the set of all operators $N\in\mathoo B(L_p)$, $1\le p\le\infty$, of the
form~\eqref{e:operator N} by $\mathoo N_1=\mathoo N_1(L_p)$.

 \begin{proposition}\label{p:n and n_1}
If two functions $n,n_1\in\mathoo N_1(\mathbb R^c,\mathbb E)$ coincide almost everywhere
on $\mathbb R^c\times\mathbb R^c$, then they induce the same operator~\eqref{e:operator
N}.
 \end{proposition}
 \begin{proof}
Indeed, for any $u\in\mathscr L_p$, by Proposition~\ref{p:5.4.3} and
Corollary~\ref{c:Fubini}, for almost all $x\in\mathbb R^c$ the functions
$v(y)=n(x,x-y)\,u(y)$ and $v_1(y)=n_1(x,x-y)\,u(y)$ coincide almost everywhere. Therefore
$Nu$ and $N_1u$ coincide almost everywhere.
 \end{proof}

 \begin{theorem}[{\rm\cite[Theprem~5.4.7]{Kurbatov99}}\bf]\label{t:5.4.7}
Let $N\in\mathoo N_1(L_p)$, $1\le p\le\infty$. If the operator $\mathbf1+N$ is
invertible, then $(\mathbf1+N)^{-1}=\mathbf1+M$, where $M\in\mathoo N_1(L_p)$.
 \end{theorem}
A version of Theorem~\ref{t:5.4.7} for the case of infinite-dimensional $\mathbb E$ can
be found in~\cite{Kurbatov01}.

 \begin{corollary}[{\rm\cite[Corollary~5.4.8]{Kurbatov99}}\bf]\label{c:5.4.8}
Let $n\in\mathoo N_1$, and the operator $N$ be defied by~\eqref{e:operator N}. If the
operator $\mathbf1+N$ is invertible in $L_p$ for some $1\le p\le\infty$, then it is
invertible in $L_p$ for all $1\le p\le\infty$. Moreover, the kernel $m$ of the operator
$M$, where $(\mathbf1+N)^{-1}=\mathbf1+M$, does not depend on $p$.
 \end{corollary}

For any $n\in\mathoo N_1(\mathbb R^c,\mathbb E)$, we denote by $\bar n$ the function that
assigns to each $x\in\mathbb R^c$ the function $\bar n(x):\,\mathbb R^c\to\mathoo
B(\mathbb E)$ defined by the rule
\begin{equation}\label{e:bar n}
\bar n(x)(y)=n(x,x-y).
\end{equation}

 \begin{proposition}\label{p:bar n}
Let $n\in\mathoo N_1(\mathbb R^c,\mathbb E)$. Then formula~\eqref{e:bar n} defines a
measurable function $\bar n$ with values in $\mathscr L_1\bigl(\mathbb R^c,\mathoo
B(\mathbb E)\bigr)$ for almost all $x\in\mathbb R^c$. The function $\bar n:\,\mathbb
R^c\to L_1\bigl(\mathbb R^c,\mathoo B(\mathbb E)\bigr)$ is essentially bounded, i.e.
$\bar n\in\mathscr L_\infty\bigl(\mathbb R^c,L_1\bigl(\mathbb R^c,\mathoo B(\mathbb
E)\bigr)\bigr)$.
 \end{proposition}
 \begin{proof}
We take an arbitrary $\alpha\in\mathbb N$. We redefine $n$ by the formula $n(x,y)=0$ for
$x\notin[-\alpha,\alpha]^c$. By estimate~\eqref{e:est via beta}, the redefinition of the
function $n$ is summable. Hence, by Proposition~\ref{p:Fubini}, for almost all
$x\in[-\alpha,\alpha]^c$, the values $\bar n(x)(y)=n(x,x-y)$ are defined for almost all
$y$, $\bar n(x)$ belongs to $\mathscr L_1$, and $\Vert\bar n(x)\Vert=\int_{\mathbb
R^c}\Vert n(x,y)\Vert\,dy\le\Vert\beta\Vert_{L_1}$.
 \end{proof}

\begin{proposition}\label{p:norm in L_infty}
For any $n\in\mathoo N_1(\mathbb R^c,\mathbb E)$, the norm of the operator
$N:\,L_\infty\to L_\infty$ defined by formula~\eqref{e:operator N} satisfies the estimate
\begin{equation*}
\esssup_x\Vert\bar n(x)\Vert_{L_1}\le C\Vert N:\,L_\infty\to L_\infty\Vert,
\end{equation*}
where $C$ depends only on the norm on $\mathbb E$.
 \end{proposition}
 \begin{proof}
We set
\begin{equation*}
M=\esssup_x\Vert n(x,\cdot)\Vert_{L_1}.
\end{equation*}

We take an arbitrary $\varepsilon>0$. By assumption, there exists a measurable set
$K\subset\mathbb R^c$ such that $\mu(K)\neq0$ and
\begin{equation*}
\Vert n(x,\cdot)\Vert_{L_1}>M-\varepsilon,\qquad x\in K.
\end{equation*}
Without loss of generality, we may assume that $K\subseteq[-\alpha,\alpha]^c$ for some
$\alpha\in\mathbb N$.

By Proposition~\ref{p:bar n}, the function $\bar n$ and its restriction to
$[-\alpha,\alpha]^c$ are measurable. Consequently, by Proposition~\ref{p:mes func}, for
any $\varepsilon_1>0$ there exists a compact set $K_1\subseteq[-\alpha,\alpha]^c$ such
that the restriction of $\bar n$ to $K_1$ is continuous and
$\mu([-\alpha,\alpha]^c\setminus K_1)<\varepsilon_1$. If $\varepsilon_1$ is small enough,
$\mu(K\cap K_1)\neq0$.

Suppose that $\dim\mathbb E=1$. Let $x_0\in K\cap K_1$ be a point of density of the
set~$K\cap K_1$. Since $L_\infty$ is the conjugate space of $L_1$, there exists $u\in
L_\infty$, $\Vert u\Vert\le1$, such that
\begin{equation*}
\int_{\mathbb R^c}n(x_0,x_0-y)u(y)\,dy\ge M-2\varepsilon.
\end{equation*}
By continuity, for $x\in K\cap K_1$ close enough to $x_0$ we have
\begin{equation*}
\Biggl|\int_{\mathbb R^c}n(x,x-y)u(y)\,dy\Biggr|\ge M-3\varepsilon.
\end{equation*}
Since $\varepsilon>0$ is arbitrary, it follows that $\Vert N:\,L_\infty\to
L_\infty\Vert\ge M$.

Now we suppose that $\dim\mathbb E$ is arbitrary (but finite). We identify elements of
$\mathoo B(\mathbb E)$ with matrices $\{a_{ij}\}$ and consider in $\mathoo B(\mathbb E)$
another norm $\Vert\{a_{ij}\}\Vert_\bullet=\max_{ij}|a_{ij}|$. It is equivalent to the
initial norm on $\mathoo B(\mathbb E)$, because all norms on a finite-dimensional space
are equivalent~\cite[ch.~1, \S~2, 3, Theorem~2]{Bourbaki TVS}. Repeating the reasoning
from the above paragraph, we obtain
\begin{equation*}
\Vert N:\,L_\infty\to L_\infty\Vert\ge\esssup_x\int_{\mathbb R^c}\Vert
n(x,y)\Vert_\bullet\;dy,
\end{equation*}
which completes the proof.
 \end{proof}

For all $r>0$, we consider the function
\begin{equation*}
\bar n_r(x)=\frac1{\mu\bigl(B(0,r)\bigr)}\int_{B(0,r)}\bar n(x-y)\,dy.
\end{equation*}
By Proposition~\ref{p:bar n}, the function $\bar n$ is essentially bounded. Therefore,
the functions $\bar n_r$ are defined everywhere and continuous.

 \begin{proposition}\label{p:bar n_r}
Let $\bar n\in\mathscr L_\infty\bigl(\mathbb R^c,L_1\bigl(\mathbb R^c,\mathoo B(\mathbb
E)\bigr)\bigr)$. Then there exists a sequence $r_i\to0$ such that the functions $\bar
n_{r_i}$ converges almost everywhere to~$\bar n$.
 \end{proposition}
 \begin{remark}\label{r:Stein}
For a locally summable function $\bar n$ taking its values in $\mathbb R$, it is
known~\cite[ch.~1, Corollary 1 of Theorem 1]{Stein} that the whole family $\bar n_{r}$
converges almost everywhere to $\bar n$ as $r\to+0$ (Lebesgue's differentiation theorem).
For our aims, the weaker assertion formulated above (which has an essentially easier
proof) is enough.
 \end{remark}

 \begin{proof}
Without loss of generality we may assume that $\bar n$ has a compact support, and thus
$\bar n\in\mathscr L_1\bigl(\mathbb R^c,L_1\bigl(\mathbb R^c,\mathoo B(\mathbb
E)\bigr)\bigr)$.

We consider the convolution operator
\begin{equation*}
\bigl(T_r\bar n\bigr)(x)=\frac1{\mu\bigl(B(0,r)\bigr)}\int_{B(0,r)}\bar n(x-y)\,dy.
\end{equation*}
It is known (see, e.g.~\cite[Theorem 4.4.4(a)]{Kurbatov99}) that
\begin{equation*}
\bigl\Vert T_r:\,L_1\bigl(\mathbb R^c,L_1\bigl(\mathbb R^c,\mathoo B(\mathbb
E)\bigr)\bigr)\to L_1\bigl(\mathbb R^c,L_1\bigl(\mathbb R^c,\mathoo B(\mathbb
E)\bigr)\bigr)\bigr\Vert\le1.
\end{equation*}

We recall that we consider the case $\bar n\in\mathscr L_1\bigl(\mathbb
R^c,L_1\bigl(\mathbb R^c,\mathoo B(\mathbb E)\bigr)\bigr)$. We take an arbitrary
$\varepsilon>0$ and a continuous function $\bar k$ with a compact support such that
$\Vert\bar k-\bar n\Vert<\varepsilon$ (the latter is possible by the
definition~\cite[ch.~4, \S~3, Definition 2]{Bourbaki Int'egration} of~$\mathscr L_1$). We
have
\begin{equation*}
\Vert T_r\bar n-\bar n\Vert_{L_1}\le\Vert T_r(\bar n-\bar k)\Vert_{L_1}+\Vert T_r\bar
k-\bar k\Vert_{L_1}+\Vert\bar k-\bar n\Vert_{L_1}\le2\varepsilon+\Vert T_r\bar k-\bar
k\Vert_{L_1}.
\end{equation*}
Since $\bar k$ is uniformly continuous, $T_r\bar k$ converges uniformly to $\bar k$ and
hence in $L_1$-norm. Thus, if $r$ is small enough, $\Vert T_r\bar k-\bar
k\Vert_{L_1}<\varepsilon$. Hence, if $r$ is small enough, we have
\begin{equation*}
\Vert\bar n_r-\bar n\Vert_{L_1}=\Vert T_r\bar n-\bar n\Vert_{L_1}\le3\varepsilon.
\end{equation*}
Thus $\bar n_r$ converges to $\bar n$ in $L_1$-norm. By Proposition~\ref{p:sequence in
L_1}, we can choose a sequence $r_i\to0$ such that $\bar n_{r_i}$ converges to $\bar n$
almost everywhere.
 \end{proof}

\section{The class $\mathoo C$}\label{e:Cu and C}
For any $\alpha\in\mathbb N$ and $1\le p\le\infty$, we denote by $L_p^\alpha$ the
subspace of $L_p$ that consists of all $u\in L_p$ such that
\begin{equation*}
u(x)=0\qquad\text{for }x\notin[-\alpha,\alpha]^c,
\end{equation*}
and we denote by $L_p^{\setminus\alpha}$ the subspace of $L_p$ that consists of all $u\in
L_p$ such that
\begin{equation*}
u(x)=0\qquad\text{for }x\in[-\alpha,\alpha]^c.
\end{equation*}
We denote by $\mathoo t_f=\mathoo t_f(L_p)$ the set of all operators $T\in\mathoo B(L_p)$
possessing the property: for any $\alpha\in\mathbb N$ there exists $\gamma\in\mathbb N$
such that
\begin{align*}
TL_p^\alpha&\subseteq L_p^\gamma,\\
TL_p^{\setminus\gamma}&\subseteq L_p^{\setminus\alpha}.
\end{align*}
We denote by $\mathoo t=\mathoo t(L_p)$ the closure of $\mathoo t_f(L_p)$ in norm. The
classes $\mathoo t_f$ and $\mathoo t$ were considered in~\cite{Kurbatov90}
and~\cite[\S~5.5]{Kurbatov99}.

 \begin{proposition}\label{p:C in t}
Let $1\le p\le\infty$. Then the class $\mathoo N_1(L_p)$ is included into $\mathoo
t(L_p)$.
 \end{proposition}
 \begin{proof} For any $\delta\in\mathbb N$, we consider the operator
\begin{equation*}
\bigl(N_\delta u\bigr)(x)=\int_{x+[-\delta,\delta]^c}n(x,x-y)\,u(y)\,dy.
\end{equation*}
Clearly, $N_\delta\in\mathoo N_1$. From estimate~\eqref{e:norm of N} we have
\begin{equation*}
\Vert N_\delta-N\Vert\le\int_{\mathbb R^c\setminus[-\delta,\delta]^c}\beta(y)\,dy,
\end{equation*}
which implies $N_\delta\to N$ as $\delta\to\infty$. It remains to observe that
\begin{align*}
N_\delta L_p^\alpha&\subseteq L_p^{\alpha+\delta},\\
N_\delta L_p^{\setminus(\alpha+\delta)}&\subseteq L_p^{\setminus\alpha}.
\end{align*}
Thus, $N_\delta\in\mathoo t_f$.
 \end{proof}

Clearly, the operator
\begin{equation*}
\bigl(S_hu\bigr)(x)=u(x-h),
\end{equation*}
where $h\in\mathbb R^c$, acts in $L_p$ for all $1\le p\le\infty$ and $\Vert S_h\Vert=1$.

 \begin{proposition}\label{p:S_hNS_-h}
Let $n\in\mathoo N_1$, and the operator $N$ be defined by formula~\eqref{e:operator N}.
Then for any $h\in\mathbb R^c$ the operator $S_hNS_{-h}$ is defined by the formula
\begin{equation*}
\bigl(S_hNS_{-h}u\bigr)(x)=\int_{\mathbb R^c}n(x-h,x-y)\,u(y)\,dy.
\end{equation*}
 \end{proposition}
 \begin{proof}
One has
\begin{align*}
\bigl(NS_{-h}u\bigr)(x)&=\int_{\mathbb R^c}n(x,x-y)\,u(y+h)\,dy,\\
\bigl(S_hNS_{-h}u\bigr)(x)&=\int_{\mathbb R^c}n(x-h,x-h-y)\,u(y+h)\,dy\\
&=\int_{\mathbb R^c}n(x-h,x-y)\,u(y)\,dy.\qed
\end{align*}
\renewcommand\qed{}
 \end{proof}

We denote by $\mathoo C_u=\mathoo C_u(L_p)$ the set of all operators $T\in\mathoo B(L_p)$
such that the function
\begin{equation}\label{e:S_hTS_-h}
h\mapsto S_hTS_{-h},\qquad h\in\mathbb R^c,
\end{equation}
is continuous in norm. Clearly, if function~\eqref{e:S_hTS_-h} is continuous at zero, it
is continuous everywhere. The class $\mathoo C_u$ was considered
in~\cite[\S~5.6]{Kurbatov99}; see also~\cite{Kurbatov10}, where the case of the strongly
differentiable function~\eqref{e:S_hTS_-h} was considered.

We denote $\mathoo C=\mathoo C(L_p)$ the set of all operators $T\in\mathoo t(L_p)$ such
that the restriction of function~\eqref{e:S_hTS_-h} to $L_p^\alpha$ is continuous in norm
for all $\alpha\in\mathbb N$. The class $\mathoo C$ was discussed in~\cite{Kurbatov88}
and~\cite[\S~5.6]{Kurbatov99}.

 \begin{theorem}[{\rm\cite[Proposition~5.6.1]{Kurbatov99}}\bf]\label{t:5.6.1}
Let $1\le p\le\infty$. If an operator $T\in\mathoo C_u(L_p)$ is invertible, then
$T^{-1}\in\mathoo C_u(L_p)$.
 \end{theorem}

 \begin{theorem}[{\rm\cite[Theorem~5.6.3]{Kurbatov99}}\bf]\label{t:5.6.3}
Let $1\le p\le\infty$. If an operator $T\in\mathoo C(L_p)$ is invertible, then
$T^{-1}\in\mathoo C(L_p)$.
 \end{theorem}

\section{The class $\mathoo C\mathoo N_1$}\label{s:CN_1}
We denote by $\mathoo C_u\mathoo N_1=\mathoo C_u\mathoo N_1(\mathbb R^c,\mathbb E)$ the
class of kernels $n\in\mathoo N_1$ such that the function $n$ can be redefined on a set
of measure zero so that it becomes defined everywhere, estimate~\eqref{e:est via beta}
holds for all $x$ and $y$, and the corresponding function $x\mapsto\bar n(x)$ becomes
\emph{uniformly} continuous in the norm of $\mathscr L_1\bigl(\mathbb R^c,\mathoo
B(\mathbb E)\bigr)$.

We denote by $\mathoo C\mathoo N_1=\mathoo C\mathoo N_1(\mathbb R^c,\mathbb E)$ the class
of kernels $n\in\mathoo N_1$ such that the function $n$ can be redefined on a set of
measure zero so that it becomes defined everywhere, estimate~\eqref{e:est via beta} holds
for all $x$ and $y$, and the corresponding function $x\mapsto\bar n(x)$ becomes
continuous in the norm of $L_1\bigl(\mathbb R^c,\mathoo B(\mathbb E)\bigr)$.

 \begin{theorem}\label{t:continuity of bar n:C_u}
Let $n\in\mathoo N_1(\mathbb R^c,\mathbb E)$, and the operator $N$ be defined by
formula~\eqref{e:operator N}. If the operator $N$ belongs to $\mathoo C_u(L_\infty)$,
then $n\in\mathoo C_u\mathoo N_1$.
 \end{theorem}
 \begin{proof}
From Propositions~\ref{p:S_hNS_-h} and~\ref{p:norm in L_infty} it follows that
\begin{equation*}
\esssup_x\Vert\bar n(x-h)-\bar n(x)\Vert_{L_1}\le C\Vert S_hNS_{-h}-N:\,L_\infty\to
L_\infty\Vert.
\end{equation*}
We recall that the assumption $N\in\mathoo C_u(L_\infty)$ means that
\begin{equation*}
\forall\varepsilon>0\quad\exists R>0\quad\forall(h:\,|h|<R)\qquad \Vert
S_hNS_{-h}-N:\,L_\infty\to L_\infty\Vert<\varepsilon,
\end{equation*}
which implies
\begin{equation*}
\forall\varepsilon>0\quad\exists R>0\quad\forall(h:\,|h|<R)\qquad \esssup_{x\in\mathbb
R^c}\Vert \bar n(x-h)-\bar n(x)\Vert_{L_1}<C\varepsilon.
\end{equation*}
Next from the estimate (for the sake of definiteness we assume that $s<r$)
\begin{multline*}
\Vert\bar n_r(x)-\bar n_s(x)\Vert
 =\biggl\Vert\frac1{\mu\bigl(B(0,r)\bigr)}\int_{B(0,r)}\bar n(x-y)\,dy-
\frac1{\mu\bigl(B(0,s)\bigr)}\int_{B(0,s)}\bar n(x-z)\,dz\biggr\Vert\\
 =\biggl\Vert\frac1{\mu\bigl(B(0,r)\bigr)}\int_{B(0,r)}\bar n(x-y)\,dy-
\Bigl(\frac sr\Bigr)^c\frac1{\mu\bigl(B(0,s)\bigr)}\int_{B(0,r)}\bar n\Bigl(x-\frac sry\Bigr)\,dy\biggr\Vert\\
 =\biggl\Vert\frac1{\mu\bigl(B(0,r)\bigr)}\int_{B(0,r)}\bar n(x-y)\,dy-
\frac1{\mu\bigl(B(0,r)\bigr)}\int_{B(0,r)}\bar n\Bigl(x-\frac sry\Bigr)\,dy\biggr\Vert\\
 =\biggl\Vert\frac1{\mu\bigl(B(0,r)\bigr)}\int_{B(0,r)}\Bigl[\bar n(x-y)-
\bar n\Bigl(x-\frac sry\Bigr)\Bigr]\,dy\biggr\Vert\\
 \le \esssup_{x\in\mathbb R^c}\Vert \bar n(x-h)-\bar n(x)\Vert,
\end{multline*}
where $h=\bigl(1-\frac sr\bigr)y$ (clearly, $|h|=\bigl|\bigl(1-\frac sr\bigr)y\bigr|\le
r$ since $0<s<r$), it follows that
\begin{equation*}
\forall\varepsilon>0\quad\exists R>0\quad\forall(r,s:\,|r|,\,|s|<R)\quad\forall
x\in\mathbb R^c\qquad \Vert\bar n_r(x)-\bar n_s(x)\Vert<C\varepsilon.
\end{equation*}
Thus $\bar n_r$ converges uniformly to a function $\bar n_*:\,\mathbb R^c\to
L_1\bigl(\mathbb R^c,\mathoo B(\mathbb E)\bigr)$ as $r\to0$. Since the functions $\bar
n_r$ are continuous, the limit function $\bar n_*$ is also continuous. On the other hand,
by Proposition~\ref{p:bar n_r}, there exists a sequence $r_i\to0$ such that the functions
$\bar n_{r_i}$ converges to $\bar n$ almost everywhere. Consequently, $\bar n$ coincides
with a continuous function~$\bar n_*$ on a set $F_1$ of full measure.

Now we describe the desired redefinition $n_0:\,\mathbb R^c\times\mathbb R^c\to\mathoo
B(\mathbb E)$ of the function $n$.

First we consider the set $E$ of all points $(x,y)\in\mathbb R^c\times\mathbb R^c$ such
that estimate~\eqref{e:est via beta} does not hold. By assumption, $E$ is a set of
measure zero. We denote by $F_2$ the set of all $x\in\mathbb R^c$ such that the set
$E_x=\{\,y\in\mathbb R^c:\,(x,y)\in E\,\}$ has measure zero. By Corollary~\ref{c:Fubini},
the set $F_2$ is a set of full measure. For $x\in F_1\cap F_2$ and $y\in E_x$, we
redefine $n$ by the rule $n(x,y)=0$. So, estimate~\eqref{e:est via beta} holds for all
$y$ when $x\in F_1\cap F_2$ (we assume that $\beta$ is defined everywhere).

We set $n_0(x,y)=n(x,y)$ for $x\in F_1\cap F_2$ and all $y\in\mathbb R^c$ (for $x\notin
F_1\cap F_2$ the value $n_0(x,y)$ is yet undefined). By Corollary~\ref{c:Fubini}, $n$ and
$n_0$ coincide on a set of full measure. If we define $n_0(x,y)$ for $x\notin F_1\cap
F_2$ in an arbitrary way, $n_0$ and $n$ remain to be equivalent functions. The problem is
to make $\bar n_0$ continuous and to ensure estimate~\eqref{e:est via beta}. We note that
for any $x\in F_1\cap F_2$ the functions $\bar n(x)(y)=n(x,y)$ and $\bar
n_0(x)(y)=n_0(x,y)$ coincide.

Next we define $n_0(x,y)$ for $x\notin F_1\cap F_2$. Since $F_1\cap F_2$ is a set of full
measure, for any $x\notin F_1\cap F_2$ there exists a sequence $x_k\in F_1\cap F_2$ that
converges to $x$. By the continuity of $\bar n_*$, it follows that $\bar
n_*(x_{k})=n_*(x_k,\cdot)$ converges to $\bar n_*(x)$ in $L_1$-norm. By
Proposition~\ref{p:sequence in L_1}, this implies that there exists a subsequence $\bar
n_*(x_{k_i})$ that converges to $\bar n_*(x)$ (not only in $L_1$-norm, but also) almost
everywhere. So, we set $n_0(x,y)=\bar n_*(x)(y)$ for $y$'s such that
$\lim_{i\to\infty}\bar n_*(x_{k_i})(y)=\bar n_*(x)(y)$ (we recall that $\bar
n_*(x_{k_i})=\bar n(x_{k_i})$ since $x_{k_i}\in F_1$), and $n_0(x,y)=0$ otherwise. By the
definition of $n_0$, we have $\Vert n_0(x_{k_i},y)\Vert\le\beta(y)$ for all $y$.
Therefore $\Vert n_0(x,y)\Vert\le\beta(y)$ for almost all $y$. Finally, we redefine
$n_0(x,\cdot)$ on a set of measure zero so that the estimate $\Vert
n_0(x,y)\Vert\le\beta(y)$ holds for all $y$.
 \end{proof}

 \begin{theorem}\label{t:continuity of bar n:C}
Let $n\in\mathoo N_1(\mathbb R^c,\mathbb E)$, and the operator $N$ be defined by
formula~\eqref{e:operator N}. The operator $N$ belongs to $\mathoo C(L_\infty)$ if and
only if $n\in\mathoo C\mathoo N_1$.
 \end{theorem}

 \begin{proof}
For any $\alpha\in\mathbb N$, we consider the operator
\begin{equation*}
\bigl(N_\alpha u\bigr)(x)=\int_{[-\alpha,\alpha]^c}n(x,x-y)\,u(y)\,dy=\int_{\mathbb
R^c}n(x,x-y)\,\chi_{[-\alpha,\alpha]^c}(y)u(y)\,dy,
\end{equation*}
where $\chi_{[-\alpha,\alpha]^c}$ is the characteristic function of the set
$[-\alpha,\alpha]^c$. Since $N_\alpha$ coincides with $N$ on $L_p^\alpha$, we have
\begin{align*}
\Vert S_hN_\alpha S_{-h}-N_\alpha:\,L_p\to L_p\Vert&=\Vert S_hN_\alpha
S_{-h}-N_\alpha:\,L_p^\alpha\to L_p\Vert\\
&=\Vert S_hNS_{-h}-N:\,L_p^\alpha\to L_p\Vert,
\end{align*}
which together with $N\in\mathoo C$ implies that $N_\alpha\in\mathoo C_u$. Therefore, by
Theorem~\ref{t:continuity of bar n:C_u}, the restriction
\begin{equation*}
\bar n_\alpha(x)(y)=n(x,x-y)\,\chi_{[-\alpha,\alpha]^c}(y)
\end{equation*}
of the function $\bar n$ coincides with a continuous function almost everywhere.
By~\eqref{e:est via beta}, for almost all $x$ we have the estimate
\begin{equation*}
\Vert\bar n_\alpha(x)\Vert_{L_1}\le\int_{[-\alpha,\alpha]^c}\beta(x-y)\,dy
=\int_{x-[-\alpha,\alpha]^c}\beta(y)\,dy=\int_{x+[-\alpha,\alpha]^c}\beta(y)\,dy.
\end{equation*}

Let us take a sequence $\alpha_i\in\mathbb N$ such that $\alpha_{i+1}-\alpha_i>2$ for all
$i$. We set
\begin{equation*}
\bar n_{\alpha_{i+1}\setminus\alpha_i}(x)(y)=n(x,x-y)\,
\chi_{[-\alpha_{i+1},\alpha_{i+1}]^c\setminus[-\alpha_{i},\alpha_{i}]^c}(y).
\end{equation*}
Clearly, $\bar n_{\alpha_{i+1}\setminus\alpha_i}=\bar n_{\alpha_{i+1}}-\bar
n_{\alpha_i}$. Obviously, the function $\bar n_{\alpha_{i+1}\setminus\alpha_i}$ coincides
with a continuous one almost everywhere. We replace the functions $\bar
n_{\alpha_{i+1}\setminus\alpha_i}$ by the corresponding continuous functions.
From~\eqref{e:est via beta} for almost all $x$, it follows the estimate
\begin{align*}
\Vert\bar
n_{\alpha_{i+1}\setminus\alpha_i}(x)\Vert_{L_1}&\le\int_{[-\alpha_{i+1},\alpha_{i+1}]^c\setminus[-\alpha_{i},\alpha_{i}]^c}\beta(x-y)\,dy\\
&=\int_{x+[-\alpha_{i+1},\alpha_{i+1}]^c\setminus[-\alpha_{i},\alpha_{i}]^c}\beta(y)\,dy.
\end{align*}

Next we take an arbitrary $x_0\in\mathbb R^c$ and show that the function $\bar n$ can be
redefined on a set of measure zero so that it becomes continuous on $x_0+(-1,1)^c$. We
note that $x\in x_0+(-1,1)^c$ and $y\in
x+[-\alpha_{i+1},\alpha_{i+1}]^c\setminus[-\alpha_{i},\alpha_{i}]^c$ imply that $y\in
x_0+[-\alpha_{i+1}-1,\alpha_{i+1}+1]^c\setminus[-\alpha_{i}+1,\alpha_{i}-1]^c$. Therefore
for $x_0+(-1,1)^c$ we have
\begin{align*}
\Vert\bar n_{\alpha_{i+1}\setminus\alpha_i}(x)\Vert_{L_1}
&\le\int_{x+[-\alpha_{i+1},\alpha_{i+1}]^c\setminus[-\alpha_{i},\alpha_{i}]^c}\beta(y)\,dy\\
&\le\int_{x_0+[-\alpha_{i+1}-1,\alpha_{i+1}+1]^c\setminus[-\alpha_{i}+1,\alpha_{i}-1]^c}\beta(y)\,dy.
\end{align*}
Clearly,
\begin{equation*}
\sum_{i=1}^\infty\int_{x_0+[-\alpha_{i+1}-1,\alpha_{i+1}+1]^c\setminus[-\alpha_{i}+1,\alpha_{i}-1]^c}\beta(y)\,dy
\le2\int_{\mathbb R^c}\beta(y)\,dy<\infty.
\end{equation*}
Hence the series $\sum_{i=1}^\infty\bar n_{\alpha_{i+1}\setminus\alpha_i}$ (consisting of
continuous functions) converges uniformly on $x_0+(-1,1)^c$. Therefore, its sum is a
continuous function. Obviously, its sum coincides with $\bar n$ on $x_0+(-1,1)^c$ almost
everywhere.

Since $x_0$ is arbitrary, $\bar n$ coincides with a continuous function~$\bar n_*$ on a
set of full measure. Now the proof of the possibility of a redefinition of $n$ repeats
the corresponding part of the proof of Theorem~\ref{t:continuity of bar n:C_u}.

Let us prove the converse statement. By Propositions~\ref{p:S_hNS_-h} and
estimate~\eqref{e:norm of N}, for any $\alpha\in\mathbb N$ we have
\begin{equation*}
\Vert S_hNS_{-h}-N:\,L_\infty^\alpha\to L_\infty\Vert\le\esssup_{x\in\mathbb
R^c}\int_{[-\alpha,\alpha]^c}\Vert n(x-h,x-y)-n(x,x-y)\Vert\,dy.
\end{equation*}
We take a large $\gamma\in\mathbb N$. For $x\notin[-\gamma,\gamma]^c$, we have the
estimate
\begin{multline*}
\int_{[-\alpha,\alpha]^c}\Vert n(x-h,x-y)-n(x,x-y)\Vert\,dy\\
\le\int_{[-\alpha,\alpha]^c}\Vert n(x-h,x-y)\Vert\,dy+\int_{[-\alpha,\alpha]^c}\Vert
n(x,x-y)\Vert\,dy\\
\le2\int_{[-\alpha,\alpha]^c}\beta(x-y)\,dy =2\int_{x+[-\alpha,\alpha]^c}\beta(y)\,dy,
\end{multline*}
which is small provided $\gamma$ is large enough. For $x\in[-\gamma,\gamma]^c$, we have
the estimate
\begin{equation*}
\int_{[-\alpha,\alpha]^c}\Vert n(x-h,x-y)-n(x,x-y)\Vert\,dy\le\int_{\mathbb R^c}\Vert
n(x-h,x-y)-n(x,x-y)\Vert\,dy,
\end{equation*}
which is small provided $h$ is small, by continuity of $\bar n$.
 \end{proof}

 \begin{theorem}\label{t:fin}
Let $n\in\mathoo C\mathoo N_1$, and the operator $N\in\mathoo B(L_p)$, $1\le p\le\infty$,
be defined by formula~\eqref{e:operator N}. If the operator $\mathbf1+N$ is invertible,
then $(\mathbf1+N)^{-1}=\mathbf1+M$, where
\begin{equation}\label{e:operator M}
\bigl(Mu\bigr)(x)=\int_{\mathbb R^c}m(x,x-y)\,u(y)\,dy
\end{equation}
with $m\in\mathoo C\mathoo N_1$.
 \end{theorem}
 \begin{proof}
Let the operator $\mathbf1+N:\,L_p\to L_p$ be invertible. Then, by
Corollary~\ref{c:5.4.8}, it is invertible in $L_\infty$, and by Theorem~\ref{t:5.4.7},
$(\mathbf1+N)^{-1}=\mathbf1+M$, where $M\in\mathoo N_1(L_\infty)$. On the other hand, by
Theorem~\ref{t:5.6.3}, $(\mathbf1+N)^{-1}\in\mathoo C(L_\infty)$. Therefore by
Theorem~\ref{t:continuity of bar n:C}, $m\in\mathoo C\mathoo N_1$.
 \end{proof}

\section{The class $\mathoo h$}\label{e:h}
For every $k\in\mathbb Z^c$, we consider the operator
\begin{align*}
\bigl(P_ku\bigr)(x)=\chi_{k+(0,1]^c}(x)u(x),
\end{align*}
where $\chi_{k+(0,1]^c}$ is the characteristic function of the set
$k+(0,1]^c\subset\mathbb R^c$. We call (see~\cite[Proposition 6.1.1]{Kurbatov99}) an
operator $K\in\mathoo t(L_p)$, $1\le p\le\infty$, \emph{locally compact} if for all
$k,m\in\mathbb Z^c$, the operator $P_mTP_k$ is compact. We denote the set of all locally
compact operators $K\in\mathoo t(L_p)$ by $\mathoo h(L_p)$. Clearly, the class $\mathoo
h(L_p)$ is closed in norm.

 \begin{theorem}\label{t:CN<k}
Let $1\le p\le\infty$. Then the class $\mathoo C\mathoo N_1(L_p)$ is included into
$\mathoo h(L_p)$.
 \end{theorem}
 \begin{proof}
It is known (see, e.g.~\cite[Proposition 6.2.2]{Kurbatov99}) that an integral operator in
$L_p[a,b]$ with a continuous kernel $k(\cdot,\cdot)$ is locally compact. Consequently, an
operator of convolution with a continuous compactly supported kernel is locally compact.
By virtue of estimate~\eqref{e:norm of N}, an operator of convolution with a summable
kernel is also locally compact.

Let an operator $N\in\mathoo C\mathoo N_1(L_p)$ has the form~\eqref{e:operator N}. Since
we want to prove the compactness of the operator $P_mTP_k$, without loss of generality we
may assume that the functions $\beta$ and $\bar n$ are compactly supported. More
precisely, we may assume that $\bar n$ is supported in $[m-1,m+2]^c$ and $\beta$ is
supported in $[m-k-1,m-k+2]^c$; moreover, the function $\bar n$ is $L_1$-continuous.

For any $i\in\mathbb N$, we consider the function
\begin{equation*}
\bar n_i(x)=\bar n(x^*),
\end{equation*}
where $x^*=(x_1^*,\dots,x_c^*)$ is the nearest to $x=(x_1,\dots,x_c)$ from the right
$\bigl(\frac1i\bigr)^c$-integer point in the sense that $0\le x_k^*-x_k<\frac1i$ and
$x^*\in\mathbb Z^c/i$. We consider the integral operators $N_i\in\mathoo N_1$ generated
by $\bar n_i$. Since $\bar n$ is continuous and compactly supported, $N_i$ converges to
$N$ in norm by estimate~\eqref{e:norm of N}.

Clearly, any operator $N_i$ can be represented as a finite sum of the operators
$P_{k,i}N_i$, $k\in\mathbb Z^c$, where
\begin{equation*}
\bigl(P_{k,i}u\bigr)(x)=\chi_{k/i+(0,1/i]^c}(x)u(x).
\end{equation*}
By what was proved, the operators $P_{k,i}N_i$ are locally compact. Hence the operators
$N_i$ and the operator $N$ are locally compact as well.
 \end{proof}

\section{The class $\mathoo C\mathoo S$}\label{s:D+N}
Let $X$ be a Banach space. We denote by $C=C(\mathbb R^c,X)$ the Banach space of all
bounded continuous functions $u:\,\mathbb R^c\to X$ with the norm
\begin{equation*}
\Vert u\Vert=\sup_{x\in\mathbb R^c}\Vert u(x)\Vert.
\end{equation*}

We denote by $\mathoo C\mathoo S=\mathoo C\mathoo S(L_p)$ the set of all operators of the
form
\begin{equation}\label{e:operator D}
\bigl(Du\bigr)(x)=\sum_{i=1}^\infty d_i(x)u(x-h_i),
\end{equation}
where $h_i\in\mathbb R^c$, $d_i\in C\bigl(\mathbb R^c,\mathoo B(\mathbb E)\bigr)$,
$\sum_{i=1}^\infty\Vert d_i\Vert_{C}<\infty$. Clearly, $D$ acts in $L_p$, $1\le
p\le\infty$, and in $C(\mathbb R^c,\mathbb E)$, and in all cases $\Vert
D\Vert\le\sum_{i=1}^\infty\Vert d_i\Vert$.

 \begin{theorem}[{\rm\cite[Corollary~5.6.10]{Kurbatov99}}\bf]\label{t:5.6.10}
If an operator $D\in\mathoo C\mathoo S$ is invertible in $L_p$, $1\le p\le\infty$, then
$D^{-1}\in\mathoo C\mathoo S$ as well. If $D\in\mathoo C\mathoo S$ is invertible in $L_p$
for some $1\le p\le\infty$, then it is invertible in $L_p$ for all $1\le p\le\infty$.
 \end{theorem}
 \begin{proof}
It is enough to observe that our class $\mathoo C\mathoo S$ coincides with the class
$\mathoo S(C)$ in notation of~\cite[see 5.2.1 and 5.1.1]{Kurbatov99}.
 \end{proof}

 \begin{proposition}\label{p:CS times CN_1}
Let $N\in\mathoo C\mathoo N_1(L_p)$ and $D\in\mathoo C\mathoo S(L_p)$, $1\le p\le\infty$.
Then $DN,ND\in\mathoo C\mathoo N_1(L_p)$.
 \end{proposition}
 \begin{proof}
Let an operator $N\in\mathoo C\mathoo N_1(L_p)$ has the form~\eqref{e:operator N}, and an
operator $D\in\mathoo C\mathoo S(L_p)$ has the form~\eqref{e:operator D}. By the
definition of composition of operators, we have
\begin{equation}\label{e:DN}
\bigl(DNu\bigr)(x)=\sum_{i=1}^\infty d_i(x)\int_{\mathbb R^c}n(x-h_i,x-y-h_i)\,u(y)\,dy.
\end{equation}

We consider the operators
\begin{align*}
\bigl(N_1u\bigr)(x)&=\int_{\mathbb R^c}\Vert n(x,x-y)\Vert\,u(y)\,dy,\\
\bigl(D_1u\bigr)(x)&=\sum_{i=1}^\infty\Vert d_i(x)\Vert u(x-h_i)
\end{align*}
acting in $L_p\bigl(\mathbb R^c,\mathbb R\bigr)$.
In the formula
\begin{equation}\label{e:D1N1}
\bigl(D_1N_1|u|\bigr)(x)=\sum_{i=1}^\infty\Vert d_i(x)\Vert\int_{\mathbb R^c}\Vert
n(x-h_i,x-y-h_i)\Vert\cdot|u(y)|\,dy,
\end{equation}
each of the functions
\begin{equation*}
v_i(x)=\int_{\mathbb R^c}\Vert n(x-h_i,x-y-h_i)\Vert\cdot|u(y)|\,dy,
\end{equation*}
by Proposition~\ref{p:5.4.3}, is defined almost everywhere and belongs to $\mathscr L_p$;
furthermore,
\begin{equation*}
\Vert v_i\Vert_{L_p}\le\Vert\beta\Vert_{L_1}\cdot\Vert u\Vert_{L_p}.
\end{equation*}

Since $\sum_{i=1}^\infty\Vert d_i\Vert_{C}<\infty$, the series $\sum_{i=1}^\infty w_i$,
where
\begin{equation*}
w_i(x)=\Vert d_i(x)\Vert\cdot v_i(x)=\Vert d_i(x)\Vert\int_{\mathbb R^c}\Vert
n(x-h_i,x-y-h_i)\Vert\cdot|u(y)|\,dy,
\end{equation*}
converges absolutely in $L_p$-norm. We denote by $F$ the set of all $x$ such that the
series $\sum_{i=1}^\infty w_i(x)$ converges (absolutely). By Proposition~\ref{p:abs ser},
$F$ is a set of full measure. Thus for $x\in F$, formulas~\eqref{e:D1N1} and~\eqref{e:DN}
can be rewritten as
\begin{equation}\label{e:DN:3}
 \begin{split}
\bigl(D_1N_1|u|\bigr)(x)&=\int_{\mathbb R^c}\sum_{i=1}^\infty\Vert d_i(x)\Vert\cdot\Vert
n(x-h_i,x-y-h_i)\Vert\cdot|u(y)|\,dy,\\
\bigl(DNu\bigr)(x)&=\int_{\mathbb R^c}\sum_{i=1}^\infty
d_i(x)\,n(x-h_i,x-y-h_i)\,u(y)\,dy
 \end{split}
\end{equation}
The later representation is obviously equivalent to
\begin{equation*}
\bigl(DNu\bigr)(x)=\int_{\mathbb R^c}\Bigl(\sum_{i=1}^\infty
d_i(x)\,n(x-h_i,x-y-h_i)\Bigr)\,u(y)\,dy.
\end{equation*}
Thus $DN$ is an integral operator with the kernel
\begin{equation*}
n_1(x,y)=\sum_{i=1}^\infty d_i(x)n(x-h_i,y-h_i).
\end{equation*}
For almost all $(x,y)$, this kernel satisfies the estimate
\begin{equation*}
\Vert n_1(x,y)\Vert\le\sum_{i=1}^\infty \Vert
d_i(x)\Vert\,\beta(y-h_i)\le\sum_{i=1}^\infty \Vert d_i\Vert_{L_\infty}\,\beta(y-h_i).
\end{equation*}
Clearly, $\beta_1(y)=\sum_{i=1}^\infty \Vert d_i\Vert_{L_\infty}\,\beta(y-h_i)$ is a
summable function. Thus $n_1\in\mathoo N_1$. It remains to note that the series
\begin{equation*}
\bar n_1(x,\cdot)=\sum_{i=1}^\infty d_i(x)\,\bar n(x-h_i,\cdot-h_i)
\end{equation*}
converges to a continuous function, because it consists of continuous functions and
converges uniformly.

Next we discuss the composition $ND$. By the definition of composition of operators, we
have
\begin{equation*}
\bigl(NDu\bigr)(x)=\int_{\mathbb R^c}n(x,x-y)\,\Bigl(\sum_{i=1}^\infty
d_i(y)u(y-h_i)\Bigr)\,dy.
\end{equation*}

Let us consider the function
$$w(y)\mapsto\sum_{i=1}^\infty \Vert
d_i(y)\Vert\cdot|u(y-h_i)|.$$ Clearly, this series converges absolutely in $L_p$-norm.
Hence by Proposition~\ref{p:abs ser} it converges absolutely almost everywhere (say, on
$F_1$) to the function $w$. Let $F_2$ be the set of all $x$ such that $\int_{\mathbb
R^c}\Vert n(x,x-y)\Vert\,w(y)\,dy<\infty$. By Proposition~\ref{p:5.4.3}, $F_2$ is a set
of full measure.

Let $x\in F_2$ be fixed. Since $n\in\mathoo C\mathoo N_1$, we may assume that $n(x,x-y)$
is defined for all $y$. Therefore for all $y\in F_1$
\begin{equation*}
n(x,x-y)\,\sum_{i=1}^\infty d_i(y)u(y-h_i)=\sum_{i=1}^\infty n(x,x-y)\,d_i(y)u(y-h_i).
\end{equation*}
Thus
\begin{align*}
\bigl(NDu\bigr)(x)&=\int_{\mathbb R^c}n(x,x-y)\,\Bigl(\sum_{i=1}^\infty
d_i(y)u(y-h_i)\Bigr)\,dy\\
&=\int_{\mathbb R^c}\Bigl(\sum_{i=1}^\infty n(x,x-y)\,d_i(y)u(y-h_i)\Bigr)\,dy.
\end{align*}
Since $x\in F_2$, the function
\begin{align*}
y&\mapsto\Vert n(x,x-y)\Vert\,w(y)\\
&=\Vert n(x,x-y)\Vert\,\sum_{i=1}^\infty\Vert d_i(y)\Vert\cdot|u(y-h_i)|\\
&=\sum_{i=1}^\infty\Vert n(x,x-y)\Vert\cdot\Vert d_i(y)\Vert\cdot|u(y-h_i)|
\end{align*}
is integrable. Therefore, by Proposition~\ref{p:Lebesgue}, the series
$$y\mapsto\sum_{i=1}^\infty n(x,x-y)\,d_i(y)\,u(y-h_i)$$
is absolutely convergent in $L_1$-norm. Hence
\begin{align*}
\bigl(NDu\bigr)(x)&=\int_{\mathbb R^c}\Bigl(\sum_{i=1}^\infty
n(x,x-y)\,d_i(y)u(y-h_i)\Bigr)\,dy\\
&=\sum_{i=1}^\infty\int_{\mathbb R^c}n(x,x-y)\,d_i(y)u(y-h_i)\,dy.
\end{align*}

Performing a simple change of variables and using the absolute convergence of the series
in $L_1$-norm, we arrive at
\begin{align*}
\bigl(NDu\bigr)(x)&=\sum_{i=1}^\infty\int_{\mathbb R^c}n(x,x-y)\,d_i(y)u(y-h_i)\,dy\\
&=\sum_{i=1}^\infty\int_{\mathbb R^c} n(x,x-y-h_i)\,d_i(y+h_i)u(y)\,dy\\
&=\int_{\mathbb R^c}\Bigl(\sum_{i=1}^\infty n(x,x-y-h_i)\,d_i(y+h_i)\Bigr)u(y)\,dy.
\end{align*}

Thus $ND$ is an integral operator with the kernel
\begin{equation*}
n_1(x,y)=\sum_{i=1}^\infty n(x,x-y-h_i)\,d_i(y+h_i).
\end{equation*}
For almost all $(x,y)$, the kernel satisfies the estimate
\begin{equation*}
\Vert n_1(x,y)\Vert\le\sum_{i=1}^\infty\beta(y+h_i)\,\Vert
d_i(y+h_i)\Vert\le\sum_{i=1}^\infty \Vert d_i\Vert_{L_\infty}\,\cdot\beta(y+h_i).
\end{equation*}
We note again that
$$\beta_1(y)=\sum_{i=1}^\infty \Vert
d_i\Vert_{L_\infty}\,\beta(y+h_i)$$ is a summable function. Thus $n_1\in\mathoo N_1$. It
remains to observe that the series
\begin{equation*}
\bar n_1(x,\cdot)=\sum_{i=1}^\infty\bar n(x,x-h_i-\cdot)\,d_i(\cdot+h_i)
\end{equation*}
consists of continuous functions and converges uniformly.
 \end{proof}

 \begin{theorem}\label{t:fin2}
Let $n\in\mathoo C\mathoo N_1$, $1\le p\le\infty$, the operator $N\in\mathoo B(L_p)$ be
defined by formula~\eqref{e:operator N}, and $D\in\mathoo C\mathoo S(L_p)$. If the
operator $D+N$ is invertible in $L_p$, then $(D+N)^{-1}=A+M$, where $A\in\mathoo C\mathoo
S(L_p)$ and $M$ has the form~\eqref{e:operator M} with $m\in\mathoo C\mathoo N_1$.
\end{theorem}
 \begin{proof}
Let $D+N$ be invertible. Then by~\cite[Theorem 6.2.1]{Kurbatov99}, the operator $D$ is
invertible. Therefore the operator $(D+N)^{-1}$ can be represented in the form
\begin{equation*}
(D+N)^{-1}=(\mathbf1+D^{-1}N)^{-1}D^{-1}.
\end{equation*}
By Theorem~\ref{t:5.6.10}, $D^{-1}\in\mathoo C\mathoo S$. By Proposition~\ref{p:CS times
CN_1}, $D^{-1}N\in\mathoo C\mathoo N_1(L_p)$. By Theorem~\ref{t:fin},
$(\mathbf1+D^{-1}N)^{-1}$ has the form $\mathbf1+K$, where $K\in\mathoo C\mathoo
N_1(L_p)$. Thus
$(D+N)^{-1}=(\mathbf1+D^{-1}N)^{-1}D^{-1}=(\mathbf1+K)D^{-1}=D^{-1}+KD^{-1}$. Finally,
again by Proposition~\ref{p:CS times CN_1}, $M=KD^{-1}\in\mathoo C\mathoo N_1(L_p)$.
 \end{proof}



\begin{thebibliography}{55}
 \bibitem{Wiener} N. Wiener. Tauberian theorems. Annals of Mathematics. \textbf{33}(1) (1932)
1--100.
 \bibitem{Bochner-Phillips} S. Bochner, R.S. Phillips. Absolutely convergent Fourier expansions for
non-commutative normed rings. Annals of Mathematics. \textbf{43}(3) (1942) 409--418.
 \bibitem{Gel'fand-Raikov-Shilov} I.M. Gel'fand, D.A. Raikov, G.E. Shilov.
Commutative Normed Rings. Fizmatgiz, Moscow, 1959 (in Russian; English transl., Chelsea
Publishing Co., New York, 1964).

 \bibitem{Aldroubi-Baskakov-Krishtal} A. Aldroubi, A.G. Baskakov, I. Krishtal. Slanted matrices, Banach frames, and
sampling. Journal of Functional Analysis. \textbf{255}(7) (2008) 1667--1691.

 \bibitem{Balan-Krishtal} R. Balan, I. Krishtal. An almost periodic noncommutative Wiener's Lemma.
Journal of Mathematical Analysis and Applications. \textbf{370}(2) (2010) 339--349.

 \bibitem{Baskakov90} A.G. Baskakov. Wiener's theorem and asymptotic estimates for elements of inverse
matrices. Funktsional. Anal. i Prilozhen. \textbf{24}(3) (1990) 64--65 (in Russian;
English translation in Functional Analysis and Its Applications. \textbf{24}(3)(1990)
222--224).

 \bibitem{Baskakov97} A.G. Baskakov. Estimates for the elements of inverse matrices, and the spectral analysis of linear
operators. Izv. Ross. Akad. Nauk, Ser. Mat. \textbf{61}(6) (1997) 3--26 (in Russian;
English translation in Izvestiya: Mathematics. \textbf{61}(6) (1997) 1113--1135).

 \bibitem{Beltita} I. Belti\c{t}\u{a}, D. Belti\c{t}\u{a}. Inverse-closed algebras of integral operators on locally compact
groups. Annales Henri Poincar\'e. Springer. (2014) 1--24.

 \bibitem{Bickel-Lindner} P. Bickel, M. Lindner. Approximating the inverse of banded matrices by banded matrices
with applications to probability and statistics. Theory of Probability \& Its
Applications. \textbf{56}(1) (2012) 1--20.

 \bibitem{Bottcher-Spitkovsky} A. B\"ottcher, I.M. Spitkovsky. Pseudodifferential operators with heavy spectrum.
Integral Equations and Operator Theory. \textbf{19}(3) (1994) 251--269.

 \bibitem{Farrell-Strohmer} B. Farrell, T. Strohmer. Inverse-closedness of a Banach algebra of integral
operators on the Heisenberg group. Journal of Operator Theory. \textbf{64}(1)(2010)
189--205.

 \bibitem{Fendler-Grochenig-Leinert} G. Fendler, K. Gr\"ochenig, M. Leinert. Convolution-dominated operators on discrete
groups. Integral Equations and Operator Theory. \textbf{61}(4) (2008) 493--509.

 \bibitem{Grochenig-Rzeszotnik-Strohmer} K. Gr\"ochenig, Z. Rzeszotnik, T. Strohmer. Quantitative estimates for the finite section
method. arXiv preprint math/0610588 (2006)

 \bibitem{Goldstein} D. Goldstein. Inverse closedness of $C^*$-algebras in Banach algebras.
Integral Equations and Operator Theory. \textbf{33}(2) (1999) 172--174.

 \bibitem{Grochenig-Leinert} K. Gr\"ochenig, M. Leinert. Symmetry and inverse-closedness of matrix algebras and
functional calculus for infinite matrices. Trans. Amer. Math. Soc. 
\textbf{358}(6) (2006) 2695--2711.

 \bibitem{Grochenig-Klotz} K. Gr\"ochenig, A. Klotz. Noncommutative approximation: inverse-closed subalgebras and
off-diagonal decay of matrices. Constructive Approximation. \textbf{32}(2) (2010)
429--466.

 \bibitem{Krishtal-Okoudjou} I.A. Krishtal, K.A. Okoudjou. Invertibility of the Gabor frame operator on the Wiener amalgam
space. Journal of Approximation Theory. \textbf{153}(2) (2008) 212--224.

 \bibitem{Kurbatov88} V.G. Kurbatov. On functional-differential equations with continuous
coefficients. Matematicheskie Zametki. \textbf{44}(6) (1988) 850--852 (in Russian).

 \bibitem{Kurbatov90} V.G. Kurbatov. The inverse to a $c$-continuous operator. Matematicheskie Zametki.
\textbf{48}(5) (1990) 68--71 (in Russian; English translation in Math. Notes.
\textbf{48}(5) (1990) 1123--1125).

 \bibitem{Kurbatov99} V.G. Kurbatov. Functional Differential Operators and Equations.
Kluwer Academic Publishers, Dordrecht, 1999.

 \bibitem{Kurbatov01} V.G. Kurbatov. Some algebras of operators majorized by a
convolution. Functional Differential Equations. \textbf{8}(3--4) (2001) 323--333.

 \bibitem{Kurbatov10} V.G. Kurbatov. On operators that are smooth with respect to an action of a
group. In: Voronezh Winter Mathematical School 2010 (2010) 126--132 (in Russian).

 \bibitem{Mantoiu} M. Mantoiu. Symmetry and inverse closedness for some Banach $C^* $-algebras associated to
discrete groups. arXiv preprint arXiv:1407.2371v1 (2014).

 \bibitem{Marius} M. Marius. Symmetry and inverse closedness for Banach $C^* $-algebras associated to discrete
groups. Banach Journal of Mathematical Analysis. \textbf{9}(2) (2015) 289--310.

 \bibitem{Sun05} Q. Sun. Wiener's lemma for infinite matrices with polynomial off-diagonal decay.
Comptes Rendus Mathematique. \textbf{340}(8) (2005) 567--570.

 \bibitem{Sun11} Q. Sun. Wiener's lemma for infinite matrices II. Constructive
Approximation, \textbf{34}(2) (2011) 209--235.

 \bibitem{Bourbaki Int'egration} N. Bourbaki. \'El\'ements de Math\'ematique. Premi\`ere Partie.
Les structures fondamentals de l'analyse. Livre VI. Int\'egration. Hermann, Paris. Chap.
1--4, 1965; Chap. 5, 1967; Chap. 6, 1959; Chap. 7--8, 1963.
 \bibitem{Edwards} R.E. Edwards. Functional Analysis. Theory and Applications. Holt, Rinehart and Winston,
New York, 1965.
 \bibitem{Hewitt-Ross} E. Hewitt, K.A. Ross. Abstract Harmonic Analysis, Vol.~1.
Springer-Verlag, Berlin--G\"ottingen--Heidelberg, 1963.
 \bibitem{Stein} E.M. Stein. Singular Integrals and Differential Properties of Functions.
Princeton University Press, Princeton, New Jersey, 1970.
 \bibitem{Bourbaki TVS} N. Bourbaki. \'El\'ements de Math\'ematique. Premi\`ere Partie. Les structures
fondamentals de l'analyse. Livre V. Espaces Vectoriels Topologiques. Hermann, Paris,
1964.
\end{thebibliography}
\end{document}